\documentclass{amsart}

\makeatletter
\@namedef{subjclassname@2020}{%
  \textup{2020} Mathematics Subject Classification}
\makeatother

\newcommand\newterm[1]{\textit{#1}}

\label{macro: brackets}
\newcommand\paren[1]{(#1)}

\label{macro: sets and maps}
\newcommand\ordpair\paren
\newcommand\restr[2]{#1\rvert_{#2}}
\newcommand\comp\circ

\label{macro: basic objects}
\newcommand\Zahlen{\mathbb Z}
\newcommand\Zap{\Zahlen_{>0}}
\newcommand\Quot{\mathbb Q}
\newcommand\Real{\mathbb R}

\label{macro:groups}
\newcommand\gengp[1]{\langle #1\rangle}

\label{macro: varieties}
\newcommand\projsp[1]{\mathbb P^{#1}}
\newcommand\isom\cong

\DeclareMathOperator\Pic{Pic}

\label{macro: divisors}

\newcommand\lineq\sim
\newcommand\lineqQ{\lineq_\Quot}
\newcommand\numeq\equiv

\DeclareMathOperator\vol{vol}

\label{braket}
\usepackage{braket}

\label{tikz}
\usepackage{tikz}
\usetikzlibrary{cd}

\label{hyperref}
\usepackage[linktocpage]{hyperref}

\label{amsthm}
\usepackage{cleveref}

\newtheorem{thm}{Theorem}[section]
\crefname{thm}{Theorem}{Theorems}
\newtheorem{prop}[thm]{Proposition}
\crefname{prop}{Proposition}{Propositions}
\newtheorem{lem}[thm]{Lemma}
\crefname{lem}{Lemma}{Lemmas}
\newtheorem{cor}[thm]{Corollary}
\crefname{cor}{Corollary}{Corollaries}
\newtheorem{conj}[thm]{Question}
\crefname{conj}{Question}{Questions}

\theoremstyle{definition}
\newtheorem{defi}[thm]{Definition}
\crefname{defi}{Definition}{Definitions}
\newtheorem{exa}[thm]{Example}
\crefname{exa}{Example}{Examples}
\newtheorem*{claim}{Calim}
\crefname{claim}{Claim}{Claims}
\newtheorem*{ack}{Acknowledgements}

\theoremstyle{remark}
\newtheorem{rem}[thm]{Remark}
\crefname{rem}{Remark}{Remarks}

\title{On boundedness of indices of minimal pairs --- surfaces}
\author{Yuto Masamura}
\address{Graduate School of Mathematical Sciences, University of Tokyo, 3-8-1 Komaba, Meguro-ku, Tokyo 153-8914, Japan}
\email{masamura@ms.u-tokyo.ac.jp}

\begin{document}

\begin{abstract}
For given positive integers $d$ and $m$, consider the projective klt pairs $\ordpair{X,B}$ of dimension $d$, of Cartier index $m$, and with semi-ample $K_X+B$ defining a contraction $\pi\colon X\to Z$.
We prove that it is not possible in general to write $n\paren{K_X+B}\lineq\pi^*A_Z$ for some $n$ depending only on $d$ and $m$, and some Cartier divisor $A_Z$ on $Z$.
\end{abstract}

\maketitle

\tableofcontents

\section{Introduction}

We work over an algebraically closed field of characteristic $0$.

In this paper, we consider the following question, which we see is not true:

\begin{conj}\label{main conj}
Let $d,m\in\Zap$.
Then there exists $n\in\Zap$ satisfying the following:
if $\ordpair{X,B}$ is a projective klt pair such that
\begin{itemize}
	\item $\dim X=d$,
	\item $m\paren{K_X+B}$ is Cartier, and
	\item $K_X+B$ is semi-ample defining a contraction $\pi\colon X\to Z$,
\end{itemize}
then
\[
	n\paren{K_X+B}\lineq\pi^*A_Z
\]
for some Cartier divisor $A_Z$ on $Z$.
\end{conj}

It is easy to see that \cref{main conj} is true if the following question, proposed by Hashizume \cite{hashi-symp}, on effective base point freeness of semi-ample log canonical divisors, is true.
Therefore it is concluded that the effective base point freeness also does not hold in general.

\begin{conj}[{\cite[Question 5.1]{hashi-symp}}] \label{eff}
Let $d,m\in\Zap$.
Then there exists $n\in\Zap$ satisfying the following:
if $\ordpair{X,B}$ is a projective klt pair such that
\begin{itemize}
	\item $\dim X=d$,
	\item $m\paren{K_X+B}$ is Cartier, and
	\item $K_X+B$ is semi-ample,
\end{itemize}
then $n\paren{K_X+B}$ is a base point free Cartier divisor.
\end{conj}

It is easy to see that \Cref{main conj} is true for $d=1$, see \cref{curve}.
Furthermore we show that \cref{main conj} is true for $d=2$ when we assume the pairs to have some property:
\begin{thm}[\cref{thm for surf}]
	\Cref{main conj} is true for $d=2$ if the contractions $\pi\colon\ordpair{X,B}\to Z$ are not of elliptic type.
\end{thm}
For the definition of \newterm{elliptic type}, see \cref{def type}.

On the other hand, we prove the following:
\begin{thm}[\cref{higher}]
 For every $d\ge2$, \cref{main conj} is not true in general, even if $X$ is smooth and $B=0$.
 In particular, for every $d\ge2$ \cref{eff} is not ture in general.
\end{thm}

The case of $\dim Z=0$ in \cref{main conj}, that is, the case $K_X+B\lineqQ0$, is closely related to the index conjecture for Calabi--Yau pairs \cite[Conjecture 1.5]{jl}.
It was studied by Jiang \cite{jiang} and Xu \cite{xu,xu2}.
It was proved in full generality in dimension at most $3$ by Jiang--Liu \cite[Corollary 1.6]{jl}, and is widely expected to hold in higher dimensions (see \cite[following Corollary 1.4]{xu}).

If $\dim Z>0$, then \cref{main conj} is true when we allow the divisor $A_Z$ to be a $\Quot$-divisor (see \cite[Proposition 8.2]{xu}).

\Cref{eff} is a generalization of \cite[Theorem 1.2]{hashi}, which adds an additional assumption on $\ordpair{X,B}$ that
\begin{quote}
	there is a $\Quot$-Cartier integral divisor $A\ge0$ on $X$ such that the volume $\vol\paren{\restr AF}>0$ is fixed for general fibres $F$ of $\pi$.
\end{quote}

\begin{ack}
I am grateful to my advisor Professor Keiji Oguiso for informing me about \cref{main conj} and for his significant support in my studies.
Furthermore, I would like to thank Professor Kenta Hashizume and Professor Yoshinori Gongyo for valuable comments on the initial version of this paper.
\end{ack}

\section{Preliminaries}

\subsection{Varieties and divisors}

We assume a \newterm{variety} to be irreducible and reduced.
For varieties $X$ and $Y$, a \newterm{contraction} $f\colon X\to Y$ is a surjective projective morphism with connected fibres.

A \newterm{divisor} on a normal variety is a finite formal sum $\sum_i d_iD_i$ of prime divisors $D_i$ with integer coefficients $d_i$.
An $\Real$-divisor $D$ is \newterm{effective}, denoted by $D\ge0$, if the coefficient of any component of $D$ is positive.

Let $\pi\colon X\to S$ be a normal variety projective over a base $S$, and $D,D'$ be $\Real$-Cartier $\Real$-divisors on $X$.
We write $D\lineq_S D'$, $D\lineq_{S,\Quot}D'$ and $D\lineq_{S,\Real}D'$ for linear, $\Quot$-linear and $\Real$-linear equivalence over $S$, respectively.
We write $D\numeq_SD'$ for numerical equivalence over $S$.
If $S$ is a point, we omit the $S$.
Note that $D\lineq_{S,\Quot}D'$ (resp. $D\lineq_{S,\Real}D'$) if and only if $D-D'\lineq_\Quot\pi^*D_S$ (resp. $D-D'\lineq_\Real\pi^*D_S$) for some $\Quot$-Cartier $\Quot$-divisor (resp. $\Real$-Cartier $\Real$-divisor) $D_S$ on $S$.

Let $\pi\colon X\to S$ be a normal variety projective over $S$.
A Cartier divisor $D$ on $X$ is \newterm{base point free} (or \newterm{free}) over $S$ if the natural map $\pi^*\pi_*\mathcal O_X\paren D\to\mathcal O_X\paren D$ is surjective.
In this case, the divisor $D$ defines a contraction $f\colon X\to Z$ with a Cartier divisor $A_Z$ ample over $S$ such that $D\lineq f^*A_Z$.

A $\Quot$-Cartier $\Quot$-divisor $D$ on $X/S$ is \newterm{semi-ample} over $S$ if $nD$ is a free Cartier divisor for some $n\in\Zap$.
In this case, $D$ defines a contraction $f\colon X\to Z$ with a $\Quot$-Cartier $\Quot$-divisor $A_Z$ ample over $S$ such that $D\lineqQ f^*A_Z$.

\subsection{Pairs and singularities}

A \newterm{pair} $\ordpair{X,B}$ consists of a normal variety $X$ and an $\Real$-divisor $B\ge0$ such that $K_X+B$ is $\Real$-Cartier.

Let $\ordpair{X,B}$ be a pair.
Take a log resolution $f\colon Y\to\ordpair{X,B}$ and write
\[
	f^*\paren{K_X+B}=K_Y+B_Y.
\]
The pair $\ordpair{X,B}$ is \newterm{Kawamata log terminal} (\newterm{klt} for short), \newterm{log canonical} (\newterm{lc}) if the $\Real$-divisor $B_Y$ has coefficients $<1$, $\le1$ respectively.

\subsection{Base point free theorem}

We introduce the base point free theorem.

\begin{thm}[{cf.~\cite[Theorem 3-1-1]{kmm}}]\label{base point free}
Let $\pi\colon\ordpair{X,B}\to S$ be a klt pair projective over a quasi-projective variety $S$.
Let $D$ be a $\pi$-nef Cartier divisor on $X$.
Assume that $n_0D-\paren{K_X+B}$ is $\pi$-nef and $\pi$-big for some $n_0\in\Zap$.
Then the divisor $nD$ is base point free over $S$ for any integer $n\gg0$.
\end{thm}

\begin{cor}\label{base point free cor}
Let $\pi\colon\ordpair{X,B}\to S$ be a klt pair projective over a quasi-projective variety $S$.
Let $D$ be a $\pi$-nef Cartier divisor on $X$.
Assume that $nD-\paren{K_X+B}$ is $\pi$-nef and $\pi$-big for some $n\in\Zap$.
Then there exist a contraction $f\colon X\to Z$ over $S$ and a Cartier divisor $A_Z$ on $Z$ ample over $S$ such that
\[
D\lineq f^*A_Z.
\]
\end{cor}

\begin{proof}
By the base point free theorem (\cref{base point free}), $nD$ is free over $S$ for any $n\gg0$.
Let $f_n\colon X\to Z_n$ be the contraction over $S$ defined by $nD$, with a Cartier divisor $A_n$ on $Z_n$ ample over $S$ such that $f_n^*A_n\lineq nD$.
It follows that for a curve $C$ on $X$ contracted by $\pi$, the curve $C$ is contracted by $f_n$ if and only if $D\cdot C=0$.
Therefore by Zariski's main theorem, the $Z_n$ are isomorphic to each other, compatibly with both $X$ and $S$.
By identifying $Z_n$ and $Z_{n+1}$, we get
\[
D\lineq f_n^*\paren{A_{n+1}-A_n}.\qedhere
\]
\end{proof}

\section{Boundedness of indices of pairs of lower dimensions}

\subsection{Boundedness of indices of curves}

We see that \cref{main conj} is true for curves and can be extended to a more general setting.
For basic properties of curves, we refer to \cite[Chapter IV]{har}.

\begin{prop}\label{curve}
Let $m\in\Zap$ and let $\pi\colon\ordpair{X,B}\to Z$ be a contraction from a projective pair $\ordpair{X,B}$ such that
\begin{itemize}
	\item $\dim X=1$,
	\item $m\paren{K_X+B}$ is Cartier, and
	\item $K_X+B\numeq_Z0$.
\end{itemize}
Then
\[
	m\paren{K_X+B}\lineq\pi^*D_Z
\]
for some Cartier divisor $D_Z$ on $Z$.
\end{prop}

\begin{proof}
Note that $X$ is a smooth projective curve.
If $\dim Z=1$, then the contraction $\pi\colon X\to Z$ is an isomorphism, so the theorem is clear.

Assume $\dim Z=0$.
Then $K_X+B\numeq0$, and we have to show that $m\paren{K_X+B}\lineq0$.
Since $\deg\paren{K_X+B}=0$ and $B\ge0$, we see that $X$ is either a rational curve or an elliptic curve.

First assume $X$ is a rational curve.
In this case the Picard group $\Pic X\isom\Zahlen$ is torsion-free, so $m\paren{K_X+B}$ Cartier implies that $m\paren{K_X+B}\lineq0$.

Assume $X$ is an elliptic curve.
Then we have $K_X\lineq0$ and $B=0$.
Therefore $K_X+B=K_X\lineq0$, and in particular $m\paren{K_X+B}\lineq0$.
\end{proof}

\subsection{Boundedness of indices of surfaces of special types}

In this subsection we show that \cref{main conj} is true for surfaces of special types.

First we prepare a lemma, which is a consequence of the base point free theorem:

\begin{lem}\label{main thm when B big}
Let $\ordpair{X,B}$ be a projective klt pair, $\pi\colon X\to Z$ be a contraction and $m\in\Zap$.
Assume that $m\paren{K_X+B}$ is Cartier, $K_X+B\numeq_Z0$, and $B$ is $\pi$-big.
Then
\[
	m\paren{K_X+B}\lineq \pi^*D_Z
\]
for some Cartier divisor $D_Z$ on $Z$.
\end{lem}

\begin{proof}
Since $B$ is $\pi$-big, we can write $B=A+E$ where $E\ge0$ and $A$ is $\pi$-ample.
Since the pair $\ordpair{X,B}$ is klt, there exists a small $t>0$ such that $\ordpair{X,\paren{1-t}B+tE}$ is klt.

Apply the base point free theorem (\cref{base point free cor}) to the klt pair $\ordpair{X,\paren{1-t}B+t E}$ over $Z$ and the $\pi$-nef Cartier divisor $m\paren{K_X+B}$.
Note that
\[
	m\paren{K_X+B}-\paren{K_X+\paren{1-t}B+t E}
	\numeq_ZtA
\]
is $\pi$-ample.
Then we get a contraction $f\colon X\to Z'$ over $Z$ and a Cartier divisor $D_{Z'}$ on $Z'$ ample over $Z$ such that $m\paren{K_X+B}\lineq f^*D_{Z'}$.
Since $K_X+B\numeq_Z0$, we see that $D_{Z'}\numeq_Z0$, and therefore the zero divisor on $Z'$ is ample over $Z$.
This implies that the contraction $Z'\to Z$ is an isomorphism.
Now the lemma is proved.
\end{proof}

We define types of a contraction $\pi\colon X\to Z$ from a surface $X$, in order to simplify the statement of the theorem.

\begin{defi}\label{def type}
Let $X$ be a normal projective surface and let $\pi\colon X\to Z$ be a contraction.
\begin{enumerate}
	\item The contraction $\pi$ is said to be \newterm{of ruled type} if $\dim Z=1$ and general fibres of $\pi$ are rational curves.
	\item The contraction $\pi$ is said to be \newterm{of elliptic type} if $\dim Z=1$ and general fibres of $\pi$ are elliptic curves.
\end{enumerate}
\end{defi}

\begin{rem}\label{rem type}
Let $\ordpair{X,B}$ be a projective klt surface and let $\pi\colon X\to Z$ be a contraction with $K_X+B\numeq_Z0$ and $\dim Z=1$.
Then the contraction $\pi$ is either of ruled type or of elliptic type.
Indeed, for a general fibre $F$ of $\pi$, we have $K_F+B_F\numeq0$ and $B_F\ge0$, so $F$ is either rational or elliptic.
\end{rem}

We show that \cref{main conj} is true for surfaces whose contractions are not of elliptic type.
In fact, we have a more general result:

\begin{thm}\label{thm for surf}
Let $m\in\Zap$.
Then there exists $n\in\Zap$ satisfying the following:
if $\ordpair{X,B}$ is a projective klt surface and $\pi\colon X\to Z$ is a contraction such that
\begin{enumerate}
	\item\label{car} $m\paren{K_X+B}$ is Cartier,
	\item $K_X+B\lineq_{Z,\Quot}0$, and
	\item\label{not elliptic type} $\pi$ is not of elliptic type,
\end{enumerate}
then
\[
n\paren{K_X+B}\lineq \pi^*D_Z
\]
for some Cartier divisor $D_Z$ on $Z$.
\end{thm}

\begin{rem}
Note that in \cref{thm for surf} above, condition (\ref{not elliptic type}) is equivalent to that the contraction $\pi\colon X\to Z$ satisfies one of the following:
\begin{itemize}
	\item $\dim Z=0$,
	\item $\dim Z=2$, or
	\item $\pi$ is of ruled type, i.e., $\dim Z=1$ and  general fibres of $\pi$ are rational curves.
\end{itemize}
See \cref{rem type} for details.
\end{rem}

\begin{proof}
Let $\pi\colon\ordpair{X,B}\to Z$  be as in the theorem, i.e., a contraction from a projective klt surface $\ordpair{X,B}$ satisfying conditions (\ref{car})--(\ref{not elliptic type}).

First consider the case $\dim Z=0$.
This means that $K_X+B\lineqQ0$.
Then by \cite[Corollary 1.6]{jl}, there is $n\in\Zap$ depending only on $m$ such that
\[
n\paren{K_X+B}\lineq0.
\]
Hence the theorem holds when $\dim Z=0$.

Next assume $\dim Z=2$, that is, $\pi\colon X\to Z$ is birational.
Then since $B$ is $\pi$-big, we have
\[
m\paren{K_X+B}\lineq\pi^*D_Z
\]
for some Cartier divisor $D_Z$ on $Z$, by \cref{main thm when B big}.
Thus we can take $n=m$ in this case.

Assume $\dim Z=1$.
By condition (\ref{not elliptic type}), the contraction $\pi\colon X\to Z$ is of ruled type.
Now we claim the following:

\begin{claim}\label{claim: smooth rel min}
We may assume that $X$ is smooth and $\pi\colon X\to Z$ is relatively minimal, that is, $X$ has no $\paren{-1}$-curve contracted by $\pi$.
\end{claim}

\begin{proof}[Proof of Claim]
Consider not necessarily smooth $X$.
Take the minimal resolution $f\colon X'\to X$ of $X$.
We can see that $K_{X'}\le f^*K_X$ by the negativity lemma.
Thus if we write $f^*\paren{K_X+B}=K_{X'}+B'$, then $B'\ge0$ and therefore $\ordpair{X',B'}$ is klt.
It is easy to see that $m\paren{K_{X'}+B'}$ is Cartier, $K_{X'}+B'\lineq_{Z,\Quot}0$, and $\pi'=\pi\comp f\colon X'\to Z$ is of ruled type.
Furthermore if $n\paren{K_{X'}+B'}\lineq\paren{\pi'}^*D_Z$ for some $n$ and $D_Z$, then we have $n\paren{K_X+B}\lineq \pi^*D_Z$, so we may assume $X$ is smooth.

Assume that $X$ is smooth but has a $\paren{-1}$-curve $C$ contracted by $\pi$.
Then the contraction $\pi$ factors as
\[
\begin{tikzcd}
X\rar{f}& X_0\rar{\pi_0}& Z,
\end{tikzcd}
\]
where $f$ is the blow-down with exceptional curve $C$.
Since $K_X+B\numeq_Z0$, it follows by the negativity lemma that $K_X+B=f^*\paren{K_{X_0}+B_0}$ where $B_0=f_*B$.
Now $\ordpair{X_0,B_0}$ is klt, $m\paren{K_{X_0}+B_0}$ is Cartier, $K_{X_0}+B_0\lineq_{Z,\Quot}0$, and $\pi_0$ is of ruled type.
Moreover if $n\paren{K_{X_0}+B_0}\lineq \pi_0^*D_Z$ for some $n$ and $D_Z$, then $n\paren{K_X+B}\lineq\pi^*D_Z$.
Therefore we may assume that $X$ is relatively minimal over $Z$.
\end{proof}

Assume $X$ is smooth and relatively minimal over $Z$ in the following.
Since $\pi\colon X\to Z$ is of ruled type and $X$ is minimal over $Z$, the surface $X$ is a geometrically ruled surface over $Z$, that is, a $\projsp1$-bundle over $Z$.
Then $-K_X$ is $\pi$-ample, and therefore so is $B$.
By \cref{main thm when B big}, it follows that
\[
	m\paren{K_X+B}\lineq\pi^*D_Z
\]
for some Cartier divisor $D_Z$ on $Z$.
This proves the theorem.
\end{proof}

\section{Failure of boundedness of indices}

We see that \cref{main conj} is not true in general even in dimension $2$:

\begin{thm}\label{main thm}
Let $n\in\Zap$.
Then there exists a smooth projective surface $X$ such that
\begin{itemize}
	\item $K_X$ is semi-ample defining a contraction $\pi\colon X\to Z$, and
	\item there is no Cartier divisor $A_Z$ on $Z$ such that $nK_X\numeq\pi^*A_Z$.
\end{itemize}
In particular, for any $m,n\in\Zap$, there exists a projective klt surface $\ordpair{X,B}$ such that
\begin{itemize}
	\item $m\paren{K_X+B}$ is Cartier,
	\item $K_X+B$ is semi-ample defining a contraction $\pi\colon X\to Z$, and
	\item there is no Cartier divisor $A_Z$ on $Z$ with
		\[
			n\paren{K_X+B}\numeq\pi^* A_Z.
		\]
\end{itemize}
\end{thm}

\begin{proof}
We follow \cite[Example 4.6]{ku}.
To prove the theorem, we may assume that $n$ is even.

Choose a minimal smooth elliptic surface $\pi\colon X\to Z$ whose multiple fibres are exactly
\[
	\pi^*P_1=2C_1,\quad\pi^*P_2=4nC_2,\quad\pi^*P_3=4nC_3,
\]
where $C_i$ are prime divisors on $X$.
Such an $X$ is constructed as follows:
Let $C$ be the smooth projective model of the affine curve defined by $y^2=x^{4n}-1$.
Consider two automorphisms of $C$ defined by
\begin{align*}
	\sigma&\colon\ordpair{x,y}\mapsto\ordpair{x,-y},\\
	\tau&\colon\ordpair{x,y}\mapsto\ordpair{\zeta x,y},
\end{align*}
where $\zeta$ is a primitive $4n$-th root of unity.
Let $G$ be the group $\gengp{\sigma,\tau}$, which is isomorphic to the product $\mu_2\times\mu_{4n}$ of cyclic groups.
Let $E$ be an elliptic curve, and let $a,b\in E$ be points of order $2$ and $4n$ respectively such that $a\ne 2nb$.
Let $G$ act on $E$ by
\[
	\sigma\colon Q\mapsto Q+a,\quad \tau\colon Q\mapsto Q+b.
\]
Then the action of $G$ on $C\times E$ is free and we get a smooth elliptic surface
\[
\begin{tikzcd}
	X=\paren{C\times E}/G\rar{\pi}& C/G=Z\isom\projsp1.
\end{tikzcd}
\]
We see that $\paren{C\times E}/\gengp\sigma\to C/\gengp\sigma$ is an elliptic surface having $4n$ multiple fibres of multiplicity $2$.
The automorphism $\tau$ acts on this elliptic surface, and acts transitively on the $4n$ multiple fibres.
Furthermore the action of $\tau$ on $C/\gengp\sigma$ fixes two points on which the multiple fibres do not lie.
Therefore the surface $X\to Z$ satisfies the desired property.

By the canonical bundle formula, we have
\[
	K_X\lineq\pi^*\biggl( K_Z+M_Z+\frac12P_1+\frac{4n-1}{4n}P_2+\frac{4n-1}{4n}P_3\biggr)
\]
for some nef Cartier divisor $M_Z$.
Since $n$ is even, we see that
\begin{align*}
	&n\deg \biggl(K_Z+M_Z+\frac12P_1+\frac{4n-1}{4n}P_2+\frac{4n-1}{4n}P_3\biggr)\\
	&=n\deg\paren{K_Z+M_Z}+2n+\frac n2-\frac12
\end{align*}
is positive and is not an integer.
Therefore $K_X$ is semi-ample and defines $\pi$, and there is no Cartier divisor $A_Z$ on $Z$ such that $nK_X\numeq\pi^*A_Z$.
\end{proof}

We give other constructions of surfaces $\ordpair{X,B}$ as in \cref{main thm} in the following two examples.

\begin{exa}
Let $m,n\in\Zap$.
Assume $m\ge2$ and $n=mm'$ for some $m'\in\Zap$.
We construct a projective klt surface $\ordpair{X,B}$ satisfying the properties in \cref{main thm}.

Take a minimal smooth elliptic surface $\pi\colon X\to Z$ that has a multiple fibre
\[
	\pi^*P=nC
\]
where $C$ is a smooth irreducible curve on $X$, and such that the multiplicity of any singular fibre of $\pi$ divides $n$.
Indeed, we can construct such an $X$ as follows:
Let $E$ be an elliptic curve, and let $\mu_{n}$ be the cyclic group of order $n$.
Let $\mu_{n}$ act on $E$ by translation by an element of order $n$, and act on $\projsp1$ by $x\mapsto\zeta x$, where $x$ is the non-homogeneous coordinate of $\projsp1$ and $\zeta$ is a primitive $n$-th root of unity.
We then let
\[
	X=\paren{E\times\projsp1}/\mu_n.
\]

By the canonical bundle formula, we can write $K_X\lineq\pi^*D_Z$ for some $\Quot$-divisor $D_Z$ on $Z$.
It follows that $nD_Z$ is integral by the assumption on multiplicities of singular fibres of $\pi$.
Choose sufficiently many general points $P_1,\dotsc,P_l\in Z$ so that $D_Z+\sum_i P_i/m$ is ample and the pair
\[
	\biggl(X,\frac1m\sum_i\pi^*P_i+\frac1mC\biggr)=:\ordpair{X,B}
\]
is klt (since $m\ge2$).
Then $m\paren{K_X+B}$ is Cartier and $K_X+B$ is semi-ample defining $\pi$.
Furthermore, since $n=mm'$, we have
\[
	n\paren{K_X+B}
	\lineq\pi^*\biggl(nD_Z+m'\sum_iP_i+\frac1mP\biggr).
\]
Since the degree of $nD_Z+m'\sum_iP_i+P/m$ is not an integer (since $m\ge2$), we cannot write
\[
	n\paren{K_X+B}\numeq\pi^*A_Z
\]
for any Cartier divisor $A_Z$ on $Z$.
\end{exa}

\begin{exa}
Let $m,n\in\Zap$.
Assume $m\ge2$ and $n=mm'$ for some $m'\in\Zap$.
We construct, following Hashizume \cite{hashi-symp}, a projective klt surface $\ordpair{X,B}$ satisfying the properties in \cref{main thm}.

By his example \cite[Example 5.5]{hashi-symp} (see also \cite[Example 3.1]{birkar}), there exist a smooth projective surface $X$ and a contraction $\pi\colon X\to\projsp1=Z$ such that $-2nK_X\lineq\pi^*P$ ($P\in Z$  a closed point).
This is constructed by blowing up $\projsp2$ at nine points.
Now choose a general point $Q\in Z$ so that $\ordpair{X,\pi^*Q/m}=\ordpair{X,B}$ is klt.
Then $m\paren{K_X+B}$ is Cartier and $K_X+B$ is semi-ample defining the contraction $\pi$, but since
\[
	n\paren{K_X+B}\lineq\pi^*\biggl(-\frac12P+m'Q\biggr),
\]
there is no Cartier divisor $A_Z$ on $Z$ such that $n\paren{K_X+B}\numeq\pi^* A_Z$.
\end{exa}

As a corollary we see that \cref{main conj} is not true for every $d\ge2$:

\begin{cor}\label{higher}
Let $d,m,n\in\Zap$, $d\ge2$.
Then there exists a projective klt pair $\ordpair{X,B}$ such that
\begin{itemize}
	\item $\dim X=d$,
	\item $m\paren{K_X+B}$ is Cartier,
	\item $K_X+B$ is semi-ample defining a contraction $\pi\colon X\to Z$, and
	\item there exists no Cartier divisor $A_Z$ on $Z$ with
		\[
			n\paren{K_X+B}\numeq\pi^*A_Z.
		\]
\end{itemize}
Furthermore we can assume that $X$ is smooth and $B=0$.
\end{cor}

\begin{proof}
Choose a surface $\pi_1\colon\ordpair{X_1,B_1}\to Z$ as in \cref{main thm}, that is, $\ordpair{X_1,B_1}$ is a projective klt surface such that $m\paren{K_{X_1}+B_1}$ is Cartier, $K_{X_1}+B_1$ is semi-ample defining a contraction $\pi_1$ to a curve $Z$, and there is no Cartier divisor $A_Z$ on $Z$ with $n\paren{K_{X_1}+B_1}\numeq\pi_1^*A_Z$.
Moreover choose a smooth projective variety $X_2$ of dimension $d-2$ with $K_{X_2}\lineq0$.
Let $X=X_1\times X_2$ and write
\[
	p^*\paren{K_{X_1}+B_1}=K_X+B,
\]
where $p\colon X\to X_1$ is the projection.
Then $\ordpair{X,B}$ is klt of dimension $d$, the divisor $m\paren{K_X+B}$ is Cartier, and $K_X+B$ is semi-ample defining the contraction $\pi=\pi_1\comp p\colon X\to Z$.
Now if
\[
	n\paren{K_X+B}\numeq\pi^*A_Z
\]
for some Cartier divisor $A_Z$ on $Z$, then it follows that $n\paren{K_{X_1}+B_1}\numeq\pi_1^*A_Z$, which is a contradiction.

Furthermore, we can make $X_1$ to be smooth and $B_1=0$ by \cref{main thm}.
Then it follows that $X$ is smooth and $B=0$.
\end{proof}


\providecommand{\bysame}{\leavevmode\hbox to3em{\hrulefill}\thinspace}
\providecommand{\MR}{\relax\ifhmode\unskip\space\fi MR }
\providecommand{\MRhref}[2]{%
  \href{http://www.ams.org/mathscinet-getitem?mr=#1}{#2}
}
\providecommand{\href}[2]{#2}

\end{document}